\documentclass{article}
\usepackage[utf8]{inputenc}
\usepackage{amsmath}
\usepackage{amsfonts} 
\usepackage{amsthm}
\usepackage{mathtools}
\usepackage{amssymb}
\usepackage{dirtytalk}

\usepackage{csquotes}
\usepackage{tikz-cd}
\usepackage{derivative}

\theoremstyle{plain}
\newtheorem{thm}{Theorem}[section]
\newtheorem{lem}[thm]{Lemma}
\newtheorem{obs}[thm]{Observation}
\newtheorem{prop}[thm]{Proposition}
\newtheorem{cor}[thm]{Corollary}

\newtheorem*{claim*}{Claim}

\theoremstyle{definition}
\newtheorem{defi}[thm]{Definition}
\newtheorem*{ack}{Acknowledgements}
\newtheorem*{keywords}{Key words and phrases}
\newtheorem*{subject}{Mathematical subject classification (2020)}

\theoremstyle{remark}

\AtBeginEnvironment{thm}{\setcounter{equation}{0}}
\AtBeginEnvironment{lem}{\setcounter{equation}{0}}
\AtBeginEnvironment{obs}{\setcounter{equation}{0}}
\AtBeginEnvironment{defi}{\setcounter{equation}{0}}
\AtBeginEnvironment{prop}{\setcounter{equation}{0}}
\AtBeginEnvironment{cor}{\setcounter{equation}{0}}

\AtBeginEnvironment{thm}{\setcounter{case}{0}}
\AtBeginEnvironment{lem}{\setcounter{case}{0}}
\AtBeginEnvironment{obs}{\setcounter{case}{0}}
\AtBeginEnvironment{defi}{\setcounter{case}{0}}
\AtBeginEnvironment{prop}{\setcounter{case}{0}}
\AtBeginEnvironment{cor}{\setcounter{case}{0}}

\AtBeginEnvironment{thm}{\setcounter{claim}{0}}
\AtBeginEnvironment{lem}{\setcounter{claim}{0}}
\AtBeginEnvironment{obs}{\setcounter{claim}{0}}
\AtBeginEnvironment{defi}{\setcounter{claim}{0}}
\AtBeginEnvironment{prop}{\setcounter{claim}{0}}
\AtBeginEnvironment{cor}{\setcounter{claim}{0}}

\newcommand{\R}{\mathbb{R}}
\newcommand{\C}{\mathbb{C}}
\newcommand{\s}{\mathbb{S}}

\newcommand{\supp}{\textnormal{supp}}

\newcommand{\Pt}{\textnormal{Pt}}
\newcommand{\const}{\textnormal{const}}
\newcommand{\alg}{\textnormal{alg}}
\newcommand{\sym}{\textnormal{sym}}
\newcommand{\rot}{\textnormal{rot}}
\newcommand{\id}{\textnormal{id}}
\newcommand{\SO}{\mathrm{SO}}
\newcommand{\Or}{\mathrm{O}}
\newcommand{\SU}{\mathrm{SU}}
\newcommand{\U}{\mathrm{U}}
\newcommand{\implication}[2]{
  \mbox{$\text{#1}\implies\text{#2}$}
  \enspace\ignorespaces
}
\newcommand{\north}{e}
\newcommand{\south}{-e}

\title{Algebraic homotopy classes}
\author{Juliusz Banecki}
\date{}

\AtEndDocument{%
  \par
  \textsc{Juliusz Banecki,
Faculty of Mathematics and Computer Science,
Jagiellonian University,
ul. Łojasiewicza 6, 30-348 Kraków, Poland}\\
\textit{E-mail address}: 
\texttt{juliusz.banecki@student.uj.edu.pl}
}

\begin{document}
\maketitle

\begin{abstract}
We prove several positive results regarding representation of homotopy classes of spheres and algebraic groups by regular mappings. Most importantly we show that every mapping from a sphere to an orthogonal or a unitary group is homotopic to a regular one. Furthermore we prove that algebraic homotopy classes of spheres form a subgroup of the homotopy group, and that a similar result holds also for cohomotopy groups of arbitrary varieties.
\end{abstract}

\begin{keywords}
Real algebraic variety, regular map, homotopy group, homogeneous space, unit sphere, linear algebraic group.
\end{keywords}

\begin{subject}
14P05, 14P25.
\end{subject}

\section{Introduction}\label{introduction}
In the paper we deal with \emph{real affine varieties}, defined as in \cite{Mangolte} or \cite{reaL_algebraic_geometry}. Note that in this sense the real projective space $\mathbb{P}^n(\R)$ is an affine variety. Morphisms of such varieties are called \emph{regular maps}.

The paper is concerned with the question whether every continous mapping between spheres is homotopic to a regular one. It has been open since the 1980s, and although many researchers have devoted their attention to it, there has not been much progress in finding a general solution. Examples of papers exploring the topic include \cite{Bochnak_Kucharz_1987a},\cite{Bochnak_Kucharz_1987b},\cite{Peng},\cite{Wood} and more; together they cover quite a lot of cases for which the answer to the question is known to be affirmative, as detailed in Section 5 of \cite{real_homogeneous}. Unfortunately, the existing approaches are often specialized and do not significantly contribute to the development of a theory capable of tackling the problem in full generality. The aim of the current paper is to introduce several new constructions, which are rather general and substantially improve the scope of homotopy classes which are known to be represented by regular mappings. Before presenting our results we need to introduce some notation.

We always interpret the sphere $\s^n$ as being embedded in $\R^{n+1}$ as the standard unit sphere, carrying the usual real algebraic structure. For any given $n$ we denote by $\north$ the point $(1,0,\dots,0)\in\s^n$ of the sphere. For a given mapping $f:X\rightarrow \s^n$ from a topological space $X$ into a sphere we denote by $\supp(f)$ the set $X\backslash f^{-1}(e)$.

For a pointed topological space $(X,a)$ we denote by $\pi_n(X,a)$ the set of homotopy classes of mappings from $(\s^n,e)$ to $(X,a)$ equipped with the natural group structure. If it is clear which point of $X$ is highlighted, then we will just write $\pi_n(X)$ instead. If $X$ additionally has the structure of a non-singular affine variety, then we denote by $\pi_n^\alg(X)\subset \pi_n(X)$ the set of \emph{algebraic homotopy classes}, i.e. those homotopy classes which admit regular representatives. 

Similarly, for a topological space $X$ we denote by $\pi^n(X)$ the set of homotopy classes of mappings from $X$ to $\s^n$. The set forms a group, called the $n$-th cohomotopy group of $X$, whenever $X$ is a compact smooth manifold of dimension at most $2n-2$ (defined as in \cite[Chapter IX, (5.6)]{Kosinski}). If $X$ has the structure of a non-singular affine variety, then we denote by $\pi^n_\alg(X)\subset \pi^n(X)$ the set of \emph{algebraic cohomotopy classes}, i.e. those classes which admit regular representatives.

First we turn our attention to the study of how the group structure in homotopy and cohomotopy behaves on algebraic classes, arriving at the following results:

\begin{thm}\label{thm_add_hom}
For every $n$ and $k$, the subset $\pi_n^\alg(\s^k)$ is a subgroup of the homotopy group $\pi_n(\s^k)$.
\end{thm}

\begin{thm}\label{thm_add_cohom}
Let $k>0$ and let $X$ be a non-singular compact real affine variety of dimension at most $2k-2$. Then the subset $\pi^k_\alg(X)$ is a subgroup of the cohomotopy group $\pi^k(X)$.
\end{thm}
Previously such statements were known to be true only for $k=1,3$ or $7$.

In the theory developed, for a given pair $(n,k)$, a certain subgroup $\tau_n(\s^k)$ of the algebraic homotopy group $\pi_n^\alg(\s^k)$ emerges naturally. We are able to prove that whenever $k<3n-2$ it is precisely equal to the kernel of the suspension homomorphism, hence getting:
\begin{thm}\label{thm_ker}
For $k<3n-2$ every element of the kernel of the suspension homomorphism $\Sigma:\pi_n(\s^k)\rightarrow \pi_{n+1}(\s^{k+1})$ is represented by a regular mapping.
\end{thm}
The group $\tau_n(\s^k)$ also turns out to be useful in proving the following result:
\begin{thm}\label{k-odd}
Let $k$ be odd. Then $2\pi_n(\s^k)\subset \pi_n^\alg(\s^k)$ holds for all $n$.
\end{thm}
Previously, such a result was known to hold under the additional assumption that $n<2k-1$ \cite[Theorem 2]{Bochnak_Kucharz_1987b}.

We then turn our attention to study mappings from spheres to linear algebraic groups, arriving at the following:

\begin{thm}\label{thm_groups}
For every $n$ and $k$, every continuous mapping from $\s^n$ to one of the groups $\Or(k),\SO(k),\U(k)$ or $\SU(k)$ is homotopic to a regular one.
\end{thm}

Note that this is supporting evidence for Conjecture I in \cite{real_homogeneous}. As a consequence of this result we obtain
\begin{thm}\label{thm_37}
For every $n$, every continuous mapping from $\s^n$ to $\s^k$ where $k=3$ or $k=7$ is homotopic to a regular one.
\end{thm}
Such a result was previously known only for $k=1,2$ or $4$ \cite[Theorem 1.1]{Bochnak_Kucharz_1987a}. 

For the next result, for a given natural number $p$ define $a_p:=2^{\phi(p-1)}$, where $\phi(t)$ is the number of integers $i$ such that $0<i\leq t$ and $i\equiv 0,1,2$ or $4 \mod 8$. The number $a_p$ is sometimes referred to as the \emph{$p$-th Radon-Hurwitz number} \cite{Mukai_1977}. Note however that this is a different object than the Hurwitz-Radon function referred to in \cite{Bochnak_Kucharz_1987b}. Theorem \ref{thm_groups} implies the following:
\begin{thm}\label{thm_codim}
For a given $m>0$, for any $k>m+1$ congruent to $-1$ modulo $a_{m+2}$, every continuous mapping from $\s^{k+m}$ to $\s^k$ is homotopic to a regular one.
\end{thm}
This, in some sense, solves the representation problem in stable homotopy.

Lastly we study the image of the $J$-homomorphism (defined for example in \cite[Chapter XI, Section 4]{Whitehead_book}) arriving at the following
\begin{thm}\label{thm_J}
For every $n$ and $k$, $2J(\pi_n(\SO(k)))\subset\pi_{n+k}^\alg(\s^k)$. Moreover if $n\equiv 2,4,5$ or $6\mod 8$ then $J(\pi_n(\SO(k)))\subset\pi_{n+k}^\alg(\s^k)$.
\end{thm}
Altogether, our results combined with some previously known ones cover all pairs $(n,k)$ of natural numbers satisfying $n\leq k+5$:
\begin{thm}
If $n\leq k+5$ then $\pi_n(\s^k)=\pi^\alg_n(\s^k)$.
\begin{proof}
As noted in \cite[Theorem 5.6]{real_homogeneous}, the only cases which have been left unsolved until now are the five following pairs $(n,k)$:
\begin{enumerate}
    \item $(7, 3)$ and $(8, 3)$, which are covered by Theorem \ref{thm_37},
    \item $(9, 5)$ and $(11, 6)$, which are covered by Theorem \ref{thm_ker}, because $\pi_{10}(\s^6)$ and $\pi_{12}(\s^7)$ are trivial (see \cite[Chapter XIV]{Toda_1962}),
    \item $(10, 5)$, which is also covered by Theorem \ref{thm_ker}, because $\pi_{10}(\s^5)$ is cyclic of order $2$, while $\pi_{11}(\s^6)$ is infinite cyclic, so it contains no non-trivial element of finite order (\cite[Chapter XIV]{Toda_1962}).
\end{enumerate}

\end{proof}
\end{thm}

\section{Preliminaries}
\begin{defi}
A \emph{linear real algebraic group} is a Zariski closed subgroup of the general linear group $\mathrm{GL}(n,\R)$ for some $n$. Let $G$ be such a group. A \emph{homogeneous
space} for $G$ is a real affine variety $Y$
on which $G$ acts transitively, the action $G\times Y\rightarrow Y$, $(g, y) \mapsto g\cdot y$ being a regular mapping. 
\end{defi}

Notice that in this case the variety $Y$ is necessarily non-singular.

There is the following recent strong result of \cite{real_homogeneous} which proves an equivalence between homotopy and approximation in a quite general setting:
\begin{thm}[\protect{\cite[Theorem 1.1]{real_homogeneous}}]\label{kucharz-bochnak}
Let $X$ be a non-singular real affine variety and let $Y$ be a homogenous space for some linear real algebraic group $G$. Let $f:X\rightarrow Y$ be a continuous mapping. If $f$ is homotopic to a regular mapping, then $f$ can be approximated by regular mappings in the compact-open topology.
\end{thm}
We will need the following slight strengthening of the result:
\begin{obs}\label{kucharz-bochnak_obs}
If $a$ is a given point of $X$, in Theorem \ref{kucharz-bochnak} we can furthermore assume that the approximating mappings $\widetilde{f}$ of $f$ all satisfy $\widetilde{f}(a)=f(a)$.
\begin{proof}
The mapping $G\rightarrow Y$, $g\mapsto g \cdot f(a)$ is a submersion at the neutral element of $G$. Hence if $\widetilde{f}(a)$ is close to $f(a)$ then we can choose an element $g_0\in G$ close to the identity such that $g_0\cdot f(a)=\widetilde{f}(a)$. Now $x\mapsto g_0^{-1}\cdot \widetilde{f}(x)$ is a regular approximation of $f$ satisfying the desired condition.
\end{proof}
\end{obs}

Notice that it implies the following:
\begin{cor}
Let $Y$ be a homogenous space for a linear real algebraic group $G$ and let $n$ be a natural number. The following conditions are equivalent:
\begin{enumerate}
    \item every continuous mapping from $\s^n$ to $Y$ is homotopic to a regular one,
    \item $\pi_n^\alg(Y,a)=\pi_n(Y,a)$ for every point $a\in Y$,
    \item $\pi_n^\alg(Y,a)=\pi_n(Y,a)$ for some point $a\in Y$.
\end{enumerate}
\begin{proof}
\implication{1}{2} Given a mapping $f:(\s^n,e)\rightarrow (Y,a)$ it follows from Observation \ref{kucharz-bochnak_obs} that we can find a close regular approximation $\widetilde{f}$ of $f$ satisfying $\widetilde{f}(e)=a$. The mapping $\widetilde{f}$ represents the same homotopy class as $f$ by \cite[Chapter III, Theorem 2.5]{Kosinski}.

\implication{2}{3} Trivial.

\implication{3}{1} Given a mapping $f:\s^n\rightarrow Y$ we can compose it with an element $g$ of $G$ so that $h(x):= g\cdot f(x)$ satisfies $h(e)=a$. Now it suffices to apply Condition 3 to find a regular mapping $\widetilde{h}$ homotopic to $h$ and then take $\widetilde{f}(x):= g^{-1}\cdot \widetilde{h}(x)$.
\end{proof}
\end{cor}

\section{Addition of algebraic homotopy classes}
We will need to work with the stereographic projection from the point $\south$ of the sphere to the affine space. We recall that it can be represented by the following explicit formula
\begin{align*}
    \pi:\s^n\backslash\{\south \} &\rightarrow \R^n \\
    (x_1,\dots,x_{n+1})&\mapsto\left(\frac{x_2}{1+x_1},\dots,\frac{x_{n+1}}{1+x_1}\right)
\end{align*}
with inverse given by
\begin{equation*}
    (X_1,\dots,X_n)\mapsto \left(\frac{1-\sum_i X_i^2}{1+\sum_i X_i^2 }, \frac{2X_1}{1+\sum_i X_i^2},\dots,\frac{2X_n}{1+\sum_i X_i^2}\right)
\end{equation*}

All of the results in this section come as consequences of the following simple construction:
\begin{defi}
We define a mapping $\oplus:\left(\s^n\times \s^n\right)\backslash\{(\south,\south)\} \rightarrow \s^n$ in the following way:

If either $a=\south$ or $b=\south$ then $a\oplus b:=\south$. Otherwise $a\oplus b$ is defined by the following composition
\begin{center}
    \begin{tikzcd}[column sep=large]
        \s^n\backslash\{\south \}\times\s^n\backslash\{\south\} \arrow[r, "\pi\times\pi"] & \R^n\times\R^n \arrow[r, "+"] & \R^n \arrow[r, "\pi^{-1}"] & \s^n
    \end{tikzcd}
\end{center}
where $+:\R^n\times\R^n\rightarrow\R^n$ is the usual vector addition.
\end{defi}

\begin{obs}
The mapping $\oplus$ is regular.
\begin{proof}
Let the coordinates on the two spheres be denoted by $x_1,\dots,x_{n+1}$ and $y_1,\dots,y_{n+1}$ respectively. Consider the following identity in the ring of regular functions of the product $\s^n\backslash\{\south \}\times\s^n\backslash\{\south\}$:
\begin{multline}\label{identity1}
    \sum_{i=2}^{n+1}\left(\frac{x_i}{1+x_1}+\frac{y_i}{1+y_1}\right)^2=\\
    =\frac{ (1-x_1^2)(1+y_1)^2
    +2\sum_{i=2}^{n+1}x_iy_i(1+x_1)(1+y_1)+(1-y_1^2)(1+x_1)^2}{(1+x_1)^2(1+y_1)^2}=\\
    =\frac{2-2y_1x_1+2\sum_{i=2}^{n+1}x_iy_i}{(1+x_1)(1+y_1)}
\end{multline}
Here in the first equality we have used the fact that $\sum_{i=1}^{n+1} x_i^2=\sum_{i=1}^{n+1} y_i^2=1$. Plugging everything in in the definition of $\oplus$ and applying the identity (\ref{identity1}) we get that the first coordinate of $\oplus$ on $\s^n\backslash\{\south \}\times\s^n\backslash\{\south\}$ can be written as
\begin{equation*}
    \frac{(1+x_1)(1+y_1)-2+2y_1x_1-2\sum_{i=2}^{n+1}x_iy_i}{(1+x_1)(1+y_1)+2-2y_1x_1+2\sum_{i=2}^{n+1}x_iy_i}
\end{equation*}
while the $j$-th one for $j>1$ can be written as
\begin{equation*}
    \frac{2x_j(1+y_1)+2y_j(1+x_1)}{(1+x_1)(1+y_1)+2-2y_1x_1+2\sum_{i=2}^{n+1}x_iy_i}
\end{equation*}

Plugging either $a=\south$ or $b=\south$ into the these formulas one easily verifies that they are defined on the entire set $\left(\s^n\times \s^n\right)\backslash\{(\south,\south)\}$ and give a regular representation of $\oplus$. 
\end{proof}
\end{obs}

Theorems \ref{thm_add_hom} and \ref{thm_add_cohom} come naturally as consequences of the above construction:
\begin{proof}[Proof of Theorem \ref{thm_add_hom}]
Given a regular mapping $f:(\s^n,\north)\rightarrow (\s^k,\north)$, the class $-[f]\in\pi_n(\s^k)$ is represented by $f\circ \sym$, where $\sym:\s^n\rightarrow \s^n$ is any symmetry through a vector hyperplane passing through $\north$. Hence $\pi_n^\alg(\s^k)$ is closed under taking inverse elements. 

Let now $\alpha_1,\alpha_2\in \pi_n^\alg(\s^k)$ be two algebraic homotopy classes. Choose $f_1,f_2$ to be their representatives such that $f_1$ is constantly equal to $\north$ on the hemisphere $\{x_2\leq 0\}$, while $f_2$ is constantly equal to $\north$ on the hemisphere $\{x_2\geq 0\}$. Then $f_1\oplus f_2$ is well defined and by definition represents the homotopy class $\alpha_1+\alpha_2$. Now applying Observation $\ref{kucharz-bochnak_obs}$ we can find regular approximations $\widetilde{f_1},\widetilde{f_2}$ of $f_1,f_2$ respectively satisfying $\widetilde{f_1}(\north)=\widetilde{f_2}(\north)=\north$. If the approximations are close enough then $\widetilde{f_1}\oplus\widetilde{f_2}$ is well defined and close to $f_1\oplus f_2$, hence it represents the same homotopy class.
\end{proof}

\begin{proof}[Proof of Theorem \ref{thm_add_cohom}]
Similarly to the preceding proof, given a regular mapping $f:X\rightarrow \s^k$, the class $-[f]\in\pi^k(X)$ is represented by $\sym\circ f$, where $\sym:\s^k\rightarrow \s^k$ is any symmetry through a vector hyperplane. This shows that $\pi^k_\alg(X)$ is closed under taking inverse elements. 

Let now $\alpha_1,\alpha_2\in \pi^k_\alg(X)$ be two algebraic cohomotopy classes. Let $f_1,f_2$ be their representatives such that $\supp(f_1)\cap \supp(f_2)=\emptyset$ (we can always find such representatives because the codimension is large enough). Then $f_1\oplus f_2$ is well defined and by definition represents the class $\alpha_1+\alpha_2$. Now as before applying Theorem $\ref{kucharz-bochnak}$ we can find regular approximations $\widetilde{f_1},\widetilde{f_2}$ of $f_1,f_2$ respectively. If the approximations are close enough then $\widetilde{f_1}\oplus\widetilde{f_2}$ is well defined and represents the class $\alpha_1+\alpha_2$.
\end{proof}

Our next results will make use of the Pontryagin–Thom construction. Recall that there is a canonical isomorphism 
\begin{equation*}
    \Pt:\pi_n(\s^k)\rightarrow F^k(\R^n)
\end{equation*}
where $F^k(\R^n)$ is the group of framed cobordism classes of framed submanifolds of $\R^n$ of codimension $k$, and $\R^n$ is identified with $\s^n\backslash\{\north\}$. The trivial element of the group is represented by the empty manifold. For a formal treatment of this notion see \cite[Chapter IX, Section 5]{Kosinski}.

It will be convenient to introduce some notation related to this construction. A framed submanifold of $\R^n$ of codimension $k$ is a pair $(M, F)$, where $M$ is a codimension $k$ smooth submanifold of $\R^n$ with trivial normal bundle $\nu (M)$, while $F = (\nu_1,\dots , \nu_k)$ is a $k$-tuple of smooth sections of $\nu(M)$ such that $(\nu_1(a),\dots, \nu_k(a))$ is a basis of the fiber $\nu(M)_a$ of $\nu(M)$ over $a$ for every point $a$ in $M$. Sometimes when it is clear from the context what framing $M$ is equipped with we will just write $M$ instead for $(M,F)$.

\begin{prop}\label{tau_is_regular}
Let $(M_1,F_1),(M_2,F_2)$ be two framed submanifolds of $\R^n$ of codimension $k$, both framed cobordant to the empty set. Suppose that $M_1\cap M_2=\emptyset$. Then $\Pt^{-1}[(M_1\sqcup M_2,F)]$ is represented by a regular mapping, where $F$ is defined as 
\begin{equation*}
    F(a)=
    \begin{cases}
        F_1(a) &\text{ for } a\in M_1 \\
        F_2(a) &\text{ for } a\in M_2
    \end{cases}
\end{equation*}
\begin{proof}
Using the Pontryagin construction we can construct smooth mappings $f_1,f_2$, such that $f_1,f_2$ are transverse to $\{\south\}$, $f_i^{-1}(\south)=M_i$ with the framing, $f_1(\north)=f_2(\north)=\north$ and $\supp(f_1)\cap\supp(f_2)=\emptyset$. Then $f_1\oplus f_2$ represents the desired class. Now since $\Pt([f_2])$ and $\Pt([f_2])$ are framed cobordant to the empty set, $f_1$ and $f_2$ are homotopic to the constant mapping so we can apply Theorem \ref{kucharz-bochnak} to find regular approximations $\widetilde{f_1},\widetilde{f_2}$ of $f_1,f_2$ respectively. Once again, if the approximations are close enough then $\widetilde{f_1}\oplus\widetilde{f_2}$ is well defined and close to $f_1\oplus f_2$.
\end{proof}
\end{prop}

Proposition \ref{tau_is_regular} motivates the following definition:
\begin{defi}\label{def_of_tau}
We define $\tau_n(\s^k)\subset\pi_n(\s^k)$ to consist of all the classes coming from the construction of Proposition \ref{tau_is_regular}, i.e. classes $\alpha\in\pi_n(\s^k)$ for which there exists a representative $(M,F)\in\Pt(\alpha)$ which splits into two closed submanifolds $M=M_1\sqcup M_2$, such that $(M_1,F\vert_{M_1})$ and $(M_2,F\vert_{M_2})$ are both framed cobordant to $\emptyset$.
\end{defi}

We will now study some basic properties of the sets $\tau_n(\s^k)$.
\begin{obs}
The set $\tau_n(\s^k)$ forms a subgroup of the homotopy group $\pi_n(\s^k)$.
\begin{proof}
Let $(M,F)$, $(N,G)$ be two framed submanifolds of $\R^n$ which admit splittings
\begin{align*}
    M&=M_1\sqcup M_2\\
    N&=N_1\sqcup N_2
\end{align*}
as in Definition \ref{def_of_tau}. After translation we may assume that $M$ and $N$ can be separated by a hyperplane in $\R^n$. The sum of their cobordism classes is by definition represented by $(L, H)$, where $L=M\sqcup N$ and the framing $H$ is defined by
\begin{equation*}
    H(a)=
    \begin{cases}
        F(a) &\text{ for } a\in M \\
        G(a) &\text{ for } a\in N
    \end{cases}
\end{equation*}
The manifold now splits as 
\begin{equation*}
    L=(M_1\sqcup N_1)\sqcup (M_2\sqcup N_2)
\end{equation*}
It is easy to realise that this splitting satisfies the assumptions of definition \ref{def_of_tau}.
\end{proof}
\end{obs}

\begin{obs}
The group $\tau_n(\s^k)$ is contained in the kernel of the suspension homomorphism $\Sigma:\pi_n(\s^k)\rightarrow\pi_{n+1}(\s^{k+1})$.
\begin{proof}
In the cobordism group suspension corresponds to embedding of $\R^n$ in $\R^{n+1}$. Hence if $(M,F)$ is a framed manifold which splits as $M=M_1\sqcup M_2$, then given the additional dimension the manifolds $M_1$ and $M_2$ can be separated so that $\Sigma[(M,F)]$ becomes equal to $\Sigma[(M_1,F\vert_{M_1})]+\Sigma[(M_2,F\vert_{M_2})]=0$.
\end{proof}
\end{obs}
For the notion of the Whitehead product referred to in the next observation see \cite[Chapter X, Section 7]{Whitehead_book}.
\begin{obs}
All Whitehead products are contained in $\tau_n(\s^k)$.
\begin{proof}
By the definition, given two smooth mappings
\begin{equation*}
    f:(\s^n,e)\rightarrow(\s^k,e),\;g:(\s^m,e)\rightarrow(\s^k,e)    
\end{equation*} 
the Whitehead product of their classes is represented by the following composition
\begin{center}
    \begin{tikzcd}
h:\s^{n+m-1} \arrow[r, "\phi"] & \s^n\vee\s^m \arrow[r, "f\vee g"] & \s^k
\end{tikzcd}
\end{center}
where $\phi$ is a certain attaching mapping and the wedge sum is the identification of the points $e$ on the two spheres. Requiring both $f$ and $g$ to be constant on small neighbourhoods of $\north$, and requiring $\phi$ to be smooth outside $\phi^{-1}(\north)$, we might assume that $h$ is smooth. Consider a regular value $a\in\s^k$ of $h$, different from $\north$. Its fiber under $f\vee g$ splits into two components, one contained in $\s^n$ and the other one in $\s^m$. Notices that we might right down these components as $(f\vee\const) ^{-1}(a)$ and $(\const\vee g)^{-1}(a)$, where $\const$ is the mapping constantly equal to $\north$. Then $h^{-1}(a)$ can be written as 
\begin{equation*}
    ((f\vee\const)\circ \phi) ^{-1}(a)\sqcup ((\const\vee g)\circ \phi )^{-1}(a)    
\end{equation*}
It remains to show that these two components equipped with the framing induced by $h$ are framed cobordant to the empty set. This follows from the fact that they represent Pontryagin classes of the Whitehead products $[f,\const]$ and $[\const,g]$ which are homotopically trivial.
\end{proof}
\end{obs}
This observation allows us to completely characterise the groups $\tau_n(\s^k)$ in the semistable homotopy range using the following result of \cite{Whitehead}:
\begin{thm}[\protect{\cite[Corollary 5.3]{Whitehead}}]
If $n<3k-2$ then every element of the kernel of the suspension $\Sigma:\pi_n(\s^k)\rightarrow\pi_{n+1}(\s^{k+1})$ can be written as a Whitehead product of a class in $\pi_{n-k+1}(\s^k)$ and the identity $[\textnormal{id}_{\s^k}]\in \pi_{k}(\s^k)$.
\end{thm}

\begin{cor}
The group $\tau_n(\s^k)$ is equal to the kernel of the suspension whenever $n<3k-2$, in particular Theorem \ref{thm_ker} holds.
\end{cor}

The author does not know how to characterise $\tau_n(\s^k)$ outside this semistable range; in particular it is not clear whether it is always equal to the kernel of the suspension or not. However, we do have the following scenario in which the group appears naturally:
\begin{lem}\label{distributivity_lemma}
Let $f:(\s^n,e)\rightarrow(\s^m,e),\;g,h:(\s^m,e)\rightarrow(\s^k,e)$. Then
\begin{equation*}
    [g\circ f]+[h\circ f]\equiv ([g]+[h])\circ [f] \mod \tau_n(\s^k)
\end{equation*}
where the composition on the right hand side is the composition product in homotopy (defined as in \cite[Chapter X, Section 8]{Whitehead_book}).
\begin{proof}
We might assume that $g$ is constantly equal to $e$ on the hemipshere $\{x_2\leq 0\}$ while $h$ is constantly equal to $e$ on the hemisphere $\{x_2\geq 0\}$ and that all three of the mappings in the lemma are smooth. The class on the right hand side of the equation is then represented by $\phi\circ f$, where $\phi$ is defined by
\begin{equation*}
    \phi(x)=\begin{cases}
        g(x) & x_2\leq 0 \\
        h(x) & x_2\geq 0
    \end{cases}
\end{equation*}
Let $a\in \s^k$ be a regular of value of the mapping $\phi\circ f$, different from $e$. Let the preimage of $a$ through $g\circ f$ with the induced framing be denoted $(M,F)$, and similarly with $(N,G)$ for the preimage through $h\circ f$. Let $(M',F')$ be the reflection of $(M,F)$ through a hyperplane far away from the origin, and similarly for $(N',G')$. Then by the definition $(M',F')\in\Pt(-[g\circ f]),(N',G')\in\Pt(-[h\circ f])$. On the other hand the preimage of $a$ through $\phi\circ f$ with the induced framing is $(M\sqcup N,H)$, where
\begin{equation*}
    H(a)=
    \begin{cases}
        F(a) &\text{ for } a\in M \\
        G(a) &\text{ for } a\in N
    \end{cases}
\end{equation*}
Finally, the Pontryagin class of $[\phi\circ f]-[g\circ f]-[h\circ f]$
is represented by $(L,H')$, where $L=M\sqcup N\sqcup M' \sqcup N'$ and
\begin{equation*}
    H'(a)=
    \begin{cases}
        H(a) &\text{ for } a\in M\sqcup N \\
        F'(a) &\text{ for } a\in M' \\
        G'(a) &\text{ for } a\in N' 
    \end{cases}
\end{equation*}
It follows that the splitting $L=(M\sqcup M')\sqcup (N\sqcup N')$ satisfies the assumptions of Definition \ref{def_of_tau}.
\end{proof}
\end{lem}

We can now deduce Theorem \ref{k-odd} from Lemma \ref{distributivity_lemma}:
\begin{proof}[Proof of Theorem \ref{k-odd}]
We consider the following classical mapping $\phi:\s^k\rightarrow\s^k$
\begin{center}
    \begin{tikzcd}
\s^k \arrow[r, hook] & \R^{k+1}\backslash \{0\} \arrow[r] & \mathbb{P}^k(\R) \arrow[r] & \s^k
\end{tikzcd}
\end{center}
where the second arrow is the projection, and the last one is the regular extension of the stereographic projection from an affine chart of the projective space. As can be easily verified, the mapping is of degree $2$, because $k$ is odd. Consider now any continuous mapping $f:(\s^n,e)\rightarrow(\s^k,e)$. The homotopy class of $\phi\circ f$ is algebraic, because we can apply the Stone-Weierstrass theorem to the dashed arrow in the following diagram:
\begin{center}
    \begin{tikzcd}
\s^n \arrow[r, "f"] \arrow[rr, dashed, bend left] & \s^k \arrow[r, hook] & \R^{k+1}\backslash\{0\} \arrow[r] & \mathbb{P}^k(\R) \arrow[r] & \s^k
\end{tikzcd}
\end{center}
Now by Lemma \ref{distributivity_lemma}:
\begin{equation*}
    [\phi\circ f]=([\text{id}]+[\text{id}])\circ [f]\equiv 2[f] \mod \tau_n(\s^k)
\end{equation*}
so
\begin{equation*}
    2[f]\in [\phi\circ f]+\tau_n(\s^k)\subset \pi_n^\alg(\s^k).
\end{equation*}
\end{proof}

\section{Mappings into linear algebraic groups}
When dealing with homotopy groups, every linear algebraic group in this section is treated as a pointed topological space with the identity matrix as the highlighted point.

The following result will play a crucial role:
\begin{thm}\label{large_m}
For every $n$, for sufficiently large $m$ every continuous mapping from $\s^n$ into $\SO(m)$ or into $\U(m)$ is homotopic to a regular one.
\begin{proof}
In the case of the unitary group this is proven in \cite[Proposition 6.1 and Theorem 6.2]{real_homogeneous}. In the case of the special orthogonal group the conclusion follows from the following approach:

From the Bott periodicity theorem we get that for $m$ large enough, the homotopy group $\pi_n(\SO(m))$ is either trivial or cyclic. In the non-trivial case it follows from \cite[Proposition 5.7]{real_homogeneous} that it suffices to only represent the generator of the homotopy group. Such a representation is established in \cite[Corollary 1.6]{Baum_1967}.
\end{proof}
\end{thm}

Consider the mapping $p:\SO(n)\rightarrow \s^{n-1}$ associating to each element of $\SO(n)$ its first column vector. Define $U^+:=\s^{n-1}\backslash\{\south\}$.

\begin{obs}\label{existance_of_section_SO}
There exists a regular mapping $s:U^+\rightarrow \SO(n)$ satisfying
\begin{align*}
    p\circ s &=\id_{U^+}\\
    s(\north) &=\textnormal{the identity matrix}
\end{align*}
\begin{proof}
The mapping can be defined in the following way: for a given point $a=(x_1,\dots,x_n)\in\s^{n-1}$ consider the two following hyperplanes in $\R^n$:
\begin{align*}
    H_1&:=\{x\in\R^n:x\perp \north\}\\
    H_2&:=\{x\in\R^n:x\perp a+e\}
\end{align*}
Then associate to $a$ the isometry
\begin{equation*}
    \sym_{H_2}\circ \sym_{H_1}
\end{equation*}
where $\sym_{H_i}$ is the symmmetry through the hyperplane $H_i$ for $i=1,2$.

Explicitly in the matrix form this isometry is given by the following formula: its $(i,j)$-th entry is equal to
\begin{equation*}
    \alpha_{i,j}:=
    \begin{cases}
        x_i & \textnormal{ if }j=1 \\
        -x_j & \textnormal{ if }i=1,\;j\neq 1 \\
        1-\frac{x_i^2}{1+x_1} & \textnormal{ if }i=j,\;i,j>1 \\
        -\frac{x_ix_j}{1+x_1} & \textnormal{ if }i\neq j,\;i,j>1 
    \end{cases}
\end{equation*}

Clearly it is a regular mapping of $a$. The conditions of the observation are verified directly from the definition.
\end{proof}
\end{obs}

Now suppose that $n$ is even, so it can be written as $n=2k$. Associating $\R^{n}$ with $\C^{k}$, $U(k)$ is naturally embedded in $\SO(n)$. In such a case the observation can be strengthened to give the following:
\begin{obs}
There exists a regular mapping $s':U^+\rightarrow \U(k)$ satisfying
\begin{align*}
    p\circ s' &=\id_{U^+}\\
    s'(\north) &=\textnormal{the identity matrix}
\end{align*}
\begin{proof}
In this case the mapping can be defined in the following way: for a given point $a=(x_1,\dots,x_k)\in\s^{n-1}\subset \C^k$ consider the two following (complex) hyperplanes in $\C^k$:
\begin{align*}
    H_1&:=\{x\in\C^k:x\perp \north\}\\
    H_2&:=\{x\in\C^k:x\perp a+e\}
\end{align*}
Then associate to $a$ the isometry
\begin{equation*}
    \rot_{H_2}\circ \sym_{H_1}
\end{equation*}
where $\sym_{H_1}$ is the symmetry through $H_1$, while $\rot_{H_2}$ is the rotation fixing $H_2$ and taking the point $\south$ to $a$.

Explicitly in the complex matrix form this isometry is given by the following formula: its $(i,j)$-th entry is equal to
\begin{equation*}
    \alpha_{i,j}:=
    \begin{cases}
        x_i & \textnormal{ if }j=1 \\
        -\frac{1+x_1}{1+\bar{x}_1}\bar{x}_j & \textnormal{ if }i=1,\;j\neq 1 \\
        1-\frac{|x_i|^2}{1+\bar{x}_1} & \textnormal{ if }i=j,\;i,j>1 \\
        -\frac{x_i\bar{x}_j}{1+\bar{x}_1} & \textnormal{ if }i\neq j,\;i,j>1 
    \end{cases}
\end{equation*}
Again, the conditions of the observation are verified easily.
\end{proof}
\end{obs}

For every $n$, we now consider $\SO(n-1)$ to be naturally embedded in $\SO(n)$ as the preimage of $\north$ through $p$. Similarly for every $n$ we consider $\U(n-1)$ as being embedded in $\U(n)$ as the preimage of $\north$ through $p$.

\begin{cor}\label{retraction}
There exists a Zariski open neighbourhood $A$ of $\SO(n-1)$ in $\SO(n)$ and a regular retraction $r:A\rightarrow \SO(n-1)$. Similarly there exists a Zariski open neighbourhood $A'$ of $\U(n-1)$ in $\U(n)$ and a regular retraction $r':A'\rightarrow \U(n-1)$.
\begin{proof}
Define $A:=\{g\in\SO(n):p(g)\neq \south\}$ and consider
\begin{equation*}
    r:A\ni g\mapsto (s\circ p(g))^\top\cdot g\in \SO(n)
\end{equation*}
where the dot denotes matrix multiplication and $\top$ denotes matrix transposition. This clearly is a well defined regular mapping. If $g\in\SO(n-1)$ then $s\circ p(g)=s(\north)=\textnormal{the identity}$, so $r(g)=g$. For a general $g$
\begin{equation*}
    p(r(g))=(s\circ p(g))^\top\cdot p(g)=\north
\end{equation*}
by the definition of $s$. This shows that $r$ is indeed a retraction.

The case of the unitary group is treated analagously. Again define $A':=\{g\in\U(n):p(g)\neq \south\}$ and consider
\begin{equation*}
    r':A'\ni g\mapsto (s'\circ p(g))^\ast\cdot g\in \U(n)
\end{equation*}
where the dot denotes matrix multiplication and the star denotes Hermitian transposition. The fact that $r'$ is a retraction can now be verified the same way as it was done for $r$.
\end{proof}
\end{cor}

\begin{prop}\label{O_algebraic}
For every $n$ and $k$, every continuous mapping from $\s^n$ to $\SO(k)$ or $\Or(k)$ is homotopic to a regular one.
\begin{proof}
For the case of the special orthogonal group consider any mapping $f:\s^k\rightarrow \SO(n)$. Using Theorem \ref{large_m} we can find such $m>k$ that every mapping from $\s^n$ to $\SO(m)$ is homotopic to a regular one. Hence using Theorem \ref{kucharz-bochnak} we can find a close regular approximation $g$ of $i\circ f$, where $i$ is the canonical embedding of $\SO(k)$ in $\SO(m)$. Consider 
\begin{equation*}
    r_{k+1}\circ \dots \circ r_m\circ g
\end{equation*}
where $r_i$ is the retraction from a neighbourhood of $\SO(i-1)$ in $\SO(i)$ onto $\SO(i-1)$ from Corollary \ref{retraction}. If $g$ is sufficiently close to $i\circ f$ then the composition makes sense and is a close regular approximation of $f$.

The claim now follows also for the orthogonal group, as it consists of two irreducible components both isomorphic (as affine varieties) to $\SO(k)$.
\end{proof}
\end{prop}

\begin{prop}\label{U_algebraic}
For every $n$ and $k$, every continuous mapping from $\s^n$ to $\U(k)$ or $\SU(k)$ is homotopic to a regular one.
\begin{proof}
The case of the unitary group is treated in the same manner as the case of the special orthogonal group.

The claim also follows for the special unitary group, because it is a regular retract of $\U(k)$ with the retraction given by
\begin{equation*}
    g\mapsto g\cdot
    \begin{bmatrix}
    \det(g)^{-1} & 0 & \cdots & 0\\
    0 & 1 & \cdots & 0 \\
    \vdots & \vdots & \ddots & \vdots \\
    0 & 0& \cdots & 1
    \end{bmatrix}
\end{equation*}
\end{proof}
\end{prop}

Propositions \ref{O_algebraic} and \ref{U_algebraic} together cover Theorem \ref{thm_groups}.

Our results regarding mappings into algebraic groups have some interesting consequences for mappings between spheres as well. Note that the mapping $p$ induces a morphism in homotopy
\begin{equation*}
    p_\ast:\pi_n(\SO(k+1))\rightarrow \pi_n(\s^k)
\end{equation*}
which quite obviously takes algebraic classes to algebraic ones. Since we have just shown that
\begin{equation*}
    \pi_n^\alg(\SO(k+1))=\pi_n(\SO(k+1))
\end{equation*}
we obtain

\begin{obs}
The image of $p_\ast$ is contained in $\pi_n^\alg(\s^k)$.
\end{obs}
This observation turns out to be extremely useful.

\begin{obs}
$p_\ast$ is surjective whenever $k=3$ or $k=7$.
\begin{proof}
It is well known that $\s^3$ and $\s^7$ are parallelizable. This is equivalent to the fact that there exists a global continuous section $\tau:\s^i\rightarrow\SO(i+1)$ of $p$ for $i=3,7$, which in turn implies that $p_\ast$ is surjective.
\end{proof}
\end{obs}

This immediately proves Theorem \ref{thm_37}.

It turns out that there exists a different, quite large family of pairs $(n,k)$ for which $p_\ast$ is surjective:

\begin{thm}
For a given $m>0$, for any $k>m+1$ congruent to $-1$ modulo $a_{m+2}$ the morphism
\begin{equation*}
    p_\ast:\pi_{k+m}(\SO(k+1))\rightarrow \pi_{k+m}(\s^k)
\end{equation*}
is a surjection, where $a_{m+2}$ is the $m+2$-th Radon-Hurwitz number defined in Section \ref{introduction}.
\begin{proof}
It suffices to apply \cite[Corollary 5]{Mukai_1977} with $s:=m+2, n:=k+m+3, r:=k+m+1$.
\end{proof}
\end{thm}
This proves Theorem \ref{thm_codim}.

We now proceed to prove our result regarding the $J$-homomorphism, it is based on the following lemma:
\begin{lem}\label{image_of_J}
Suppose that $U$ is a Zariski open neighbourhood of $\s^n$ in $\R^{n+1}$ and $f:U\rightarrow \SO(k)$ is a regular mapping taking $e$ to the identity matrix. Then $J([f\vert_{\s^n}])\in \pi_{n+k}^\alg(\s^k)$, where $J:\pi_n(\SO(k))\rightarrow \pi_{n+k}(\s^k)$ is the $J$-homomorphism.
\begin{proof}
Let $Q$ be a denominator of $f$, i.e. a regular function on $\R^{n+1}$, positive on $\s^n$, such that $f$ can be written as 
\begin{equation*}
    f=\frac{P}{Q}
\end{equation*}
where $P:\R^{n+1}\rightarrow \R^{k^2}$ is a regular mapping. Let the coordinates on $\R^{k+n+1}$ be denoted by $(x_1,\dots,x_{n+1},y_1,\dots,y_k)=(x,y)$. Consider the rational mapping
\begin{gather*}
    g: \R^{k+n+1} \dashrightarrow \R^{k+1}\\
    (x,y)\mapsto \left(\frac{Q(x)^2-||y||^2}{Q(x)^2+||y||^2},\frac{2Q(x)}{Q(x)^2+||y||^2}(f(x))^\top\cdot y \right)
\end{gather*}
It is clear that $g$ extends in a regular way to a neighbourhood of $\s^{n+k}$ in $\R^{n+k+1}$. Moreover it maps $\s^{n+k}$ into $\s^k$, as for $(x,y)\in\s^{n+k}$ we have:
\begin{equation*}
    ||g(x,y)||^2=\frac{(Q(x)^2-||y||^2)^2+4Q(x)^2||y||^2}{(Q(x)^2+||y||^2)^2}=1
\end{equation*}
We claim that $\north\in\s^k$ is a regular value of $g\vert_{\s^{n+k}}$. To see that note that the preimage of $e$ through $g\vert_{\s^{n+k}}$ is the $n$-dimensional sphere $\s^n=\{(x,0)\in\s^{n+k}\}$. Now we check that the partial derivative of $g$ with respect to a normal vector $(0,y)$ to $\s^n$ at a point $(x,0)\in\s^n$ is non-zero:
\begin{equation*}
    d g_{(x,0)}(0,y)=\left(0,\frac{2}{Q(x)}(f(x))^\top\cdot y\right)
\end{equation*}
Moreover, the framing induced by $g$ on $\s^n$ at the point $(x,0)$ is up to scalar multiplication given by column vectors of $f$, which shows that $g\vert_{\s^{n+k}}$ represents the image of $[f\vert_{\s^n}]$ in the $J$-homomorphism (see \cite[Chapter IX, (6.3)]{Kosinski} for the definition of the $J$-homorphism in the language of framed submanifolds).
\end{proof}
\end{lem}

\begin{proof}[Proof of Theorem \ref{thm_J}]
Let us begin by proving the second part. Let $f:\s^n\rightarrow\SO(k)$, where $n\equiv 2,4,5$ or $6\mod 8$. We get from Bott periodicity that if $m$ is large enough then $i\circ f$ is homotopically trivial, where $i:\SO(k)\rightarrow\SO(m)$ is the standard inclusion. Consider $g:=i\circ f\circ \psi$, where $\psi:\R^{n+1}\backslash \{0\}\rightarrow \s^{n}$ is the radial projection. As $g$ is also homotopically trivial, we get from Theorem \ref{kucharz-bochnak} that it can be approximated by a regular mapping $\widetilde{g}:\R^{n+1}\backslash \{0\}\rightarrow \SO(m)$. If the approximation is close enough, then $h:=r_{k+1}\circ\dots\circ r_m\circ \widetilde{g}$ is a regular mapping defined on a neighbourhood of the sphere, such that $h\vert_{\s^n}$ is close to $f$. Hence Lemma \ref{image_of_J} applies here, proving that $J([f])\in \pi_{n+k}^\alg(\s^k)$.

Let us now proceed to the first part. For $n\equiv 2,4,5$ or $6\mod 8$ we have just proven something stronger. For $n\equiv 0$ or $1\mod 8$, the proof follows the same lines as the in the preceding paragraph, because from Bott periodicity we get that $i\circ f$ is homotopically trivial if $m$ is large enough and $[f]\in 2\pi_n(\SO(k))$. 

The remaining case is $n\equiv 3$ or $7 \mod 8$, it is treated in a slightly different way than the former one. Since here $n$ is odd, there exists a regular mapping $\phi:\R^{n+1}\backslash \{0\}\rightarrow \s^n$ such that $\textnormal{deg}(\phi\vert_{\s^n})=2$ (as in the proof of Theorem \ref{k-odd}). If now $f:\s^n\rightarrow\SO(k)$ is regular, then $f\circ \phi$ is a regular mapping from $\R^{n+1}\backslash \{0\}$ into $\SO(k)$ representing the class $2[f]$, so again the proof is finished by an application of Lemma \ref{image_of_J}.
\end{proof}

\begin{ack}
The author was partially supported by the National Science Centre (Poland) under grant number 2022/47/B/ST1/00211.
\end{ack}

\bibliographystyle{plain}
\bibliography{references}
\nocite{reaL_algebraic_geometry}
\end{document}